\documentclass[12pt]{article}
\title{Martingale Decomposition and BSDE on Time Scales}
\usepackage{authblk}
\author{Guofeng Tang}
\affil{Zhongtai Securities Institute for Financial Studies\authorcr Shandong University \authorcr tangguofengmath@mail.sdu.edu.cn}
\date{\today}
\usepackage{graphicx}
\usepackage{amsmath,amssymb,amstext}
\usepackage{soul}
\usepackage{dsfont}
\usepackage{latexsym}
\usepackage{cite}
\usepackage{dsfont}
\usepackage{amsthm}
\usepackage{comment}

\newtheorem{theorem}{Theorem}[section]
\newtheorem{lemma}[theorem]{Lemma}
\newtheorem{definition}[theorem]{Definition}

\newtheorem{example}[theorem]{Example}
\newtheorem{remark}[theorem]{Remark}
\UseRawInputEncoding
\begin{document}
\date{}
\maketitle

\begin{abstract}
In this paper, we present martingale decomposition on time scales.
We establish the related backward stochastic dynamic equations on time scales (this paper BS$\nabla$E for short,
concerning $\nabla$-integral on time scales) which unify backward stochastic differential equations and backward stochastic difference equations.
We prove the existence and uniqueness theorem of BS$\nabla$E.
This work can be considered as a unification and a generalization of similar results in backward stochastic difference equations and backward stochastic differential equations.
\end{abstract}
\smallskip
\noindent \textbf{keywords:} Time Scales, Martingale Decomposition on time scales,
BSDE on time scales

\section{Introduction}

In 1988, Stefan Hilger \cite{Higer1990} introduced the calculus of measure chains to unify continuous and discrete analysis.
Since then, this topic has received much attention, especially on the deterministic analysis  \cite{integration}, \cite{basics}.

For the stochastic integral on time scales, Suman \cite{Doctoral} tried to define ¡°stochastic integral on time scales¡±, but he just dealt with isolated time scales.
David Grow and Suman Sanyal \cite{existence} extended the stochastic integral for Brownian motion to a general time scale.
The existence and uniqueness of the strong solutions of a certain class of stochastic dynamic equations (S$\Delta$E) was proved.
Wenqing Hu \cite{girsanov} provided an It\^{o}'s formula with respect to (w.r.t.) the Brownian Motion on time scales
and demonstrated a Girsanov's change of measure formula in the case of general time scales.
Jocad and Shiryaev\cite{limit} showed that the discrete case actually reduced to a particular case in the general continuous one.
To that effect, they considered a discrete stochastic basis $\mathcal{B}=(\Omega,\mathcal{F},\mathbf{F}=(\mathcal{F}_n)_{n\in \mathds{N}},\mathbf{P})$.
They associated to $\mathcal{B}$ a "continuous " stochastic basis $\mathcal{B}'$ as follows:
$\mathcal{B}'=(\Omega,\mathcal{F}'_t,\mathbf{F}'), ~\text{with}~ \mathcal{F}_t'=\mathcal{F}_n,~ \text{for}~ t\in[n,n+1)$,
and the discrete process $X_n$ on $\mathcal{B}$ by $X_t'=X_n, t\in \left[n,n+1\right)$.
Bohner \cite{generaltimescale} constructed a stochastic integral and stochastic differential equations on general time scales with the same method.
These integrals are $\Delta$-integration on the semi-open intervals of the form $\left[t_i,t_{i+1}\right)$.
Meanwhile, by the requirement of the predictable integrand, it is easy to consider semi-open intervals $\left(t_i,t_{i+1}\right]$.
Nguyen Huu Du and Nguyen Thanh Dieu \cite{first} defined a stochastic calculus on time scales (for the $\nabla$ case),
in which they presented the Doob-Meyer decomposition theorem for a submartingale indexed by a time scale.
They also defined the $\nabla$-stochastic integration, It\^{o}'s formula with respect to the square integral martingales.
In \cite{SDE2013}, they considered the $\nabla$-stochastic dynamic equation on time scales (S$\nabla$E),
in which they gave the conditions for the existence and uniqueness of solutions, Markovian property of the solutions.
For the Backward Stochastic calculus on time scales, there are few results.

The theory of BSDE on continuous-time is a mature field,
while the discrete counterpart BSDE was rarely studied.
For BSDE on continuous-time:
\begin{equation}
\left\{
\begin{array}{ll}
Y_t=Y_T+\int_t^Tg(u,Y_u,Z_u)du-\int_t^TZ_ud W_u,\quad t\in \mathbf{R}_{+};\\
Y_T=\xi.\\
\end{array}
\right.
\end{equation}
A solution of this equation, associated with the terminal value $\xi$ and generator $g(\omega,t,y,z)$,
is a couple of stochastic processes $(Y_t,Z_t)_{t\in \mathbf{R}_{+}}$ which satisfies the equation.
A general nonlinear BSDE was first introduced by Pardoux and Peng \cite{adaptedBsde}.

Concerning Backward Stochastic Difference Equations, the formal work focus on the order of convergence as a numerical scheme, rarely the discrete scheme itself.
Many algorithms for BSDE are based on random walk framework. This equation has a unique solution since the driven martingale has the
Predictable Representation Property (\textbf{PRP}).
Cohen \cite{cohenfinite} studied the Backward Stochastic Difference Equations on spaces related to discrete-time, which is driven by finite-state processes.
The finite-state driven process has the predictable representation property. This kind of difference equation:
\begin{equation}\label{discretebsde}
\left\{
\begin{array}{ll}
Y_{t+1}-Y_t=-g(t,Y_t,Z_t)+Z_t(M_{t+1}-M_t),\quad t\in \mathds{N};\\
Y_T=\xi, \quad T\in \mathds{N}.\\
\end{array}
\right.
\end{equation}
has a unique solution $(Y_t,Z_t)_{t\in\mathds{N}}$.
The typical example of a model based on a finite probability space is the ¡°binomial¡± model, also known as the ¡°Cox-Ross-Rubinstein¡± model in finance.

Another kind of Backward Stochastic Difference Equation is driven by increment independent processes, such as i.i.d normal distributions.
By switching from the Brownian motion on continuous-time to discrete-time, we lose the predictable representation property.
It is well known that we need to include in the formulation of the BSDE on discrete-time an additional orthogonal martingale term.
\begin{equation}
\left\{
\begin{array}{ll}
Y_{t+1}-Y_t=-g(t,Y_t,Z_t)+Z_t(W_{t+1}-W_t)+N_{T}-N_t, \quad t\in \mathds{N};\\
Y_T=\xi.\\
\end{array}
\right.
\end{equation}
By the Galtchouk-Kunita-Watanabe Theorem, the solution of this kind equation is a triple tuple $(Y,Z,N)$,
where $N_t$ is an orthogonal martingale to the integrals w.r.t the driven process $W_t$.
Bielecki \cite{conic} first studied the existence and uniqueness of the solutions by the Galtchouk-Kunita-Watanabe decomposition of discrete BSDE.
In \cite{cohengeneral}, Cohen considered BSDE in discrete time with infinitely many states, the driven process was a martingale difference process.

For the general difference equations, there are many types: implicit or explicit.
\begin{equation}\label{discretebsde1}
\left\{
\begin{array}{ll}
Y_{k+1}-Y_{k}=-g(\omega,k,Y_k,Z_k)+Z_k(W_{k+1}-W_k)+N_{k+1}-N_k\\
Y_T=\xi\\
\end{array}
\right.
\end{equation}
\begin{equation}
\left\{
\begin{array}{ll}
Y_{k+1}-Y_{k}=-g(\omega,k+1,Y_{k+1},Z_k)+Z_k(W_{k+1}-W_k)+N_{k+1}-N_k\\
Y_T=\xi\\
\end{array}
\right.
\end{equation}
\begin{equation}
\left\{
\begin{array}{ll}
Y_{T}-Y_{T-1}=-h(\omega,Y_T)+Z_{T-1}(W_T-W_{T-1})+(N_T-N_{T-1})\\
Y_{k+1}-Y_{k}=-g(\omega,k+1,Y_{k+1},Z_{k+1})+Z_k(W_{k+1}-W_k)+(N_{k+1}-N_k)\\
Y_T=\xi\\
\end{array}
\right.
\end{equation}

This paper aims to establish the existence and uniqueness of solutions to BSDEs on time scales (BS$\nabla$E for our case).
The backward stochastic dynamic equations w.r.t. Brownian motion on time scales are similar to traditional BSDEs on continuous-time driven by general
c$\grave{a}$dl$\grave{a}$g martingales \cite{BSDEdriven}.
We consider the following BS$\nabla$E:
\begin{equation}
\left\{
\begin{array}{ll}
Y_t=Y_T+\int_t^Tg(u,Y(u-),Z(u))\nabla u-\int_t^TZ(u)\nabla W_u+N_T-N_t;\quad t\in\mathds{T},\\
Y_T=\xi.\\
\end{array}
\right.
\end{equation}
The solution is given by $(Y,Z,N)$, where $N$ can not be constructed by the integrals w.r.t. the Brownian motion on time scales. We assume the usual conditions on BSDE:
$\xi\in L^2(\Omega,\mathcal{F},\mathbf{P};\mathbf{R}^n)$, and $g$ satisfies the Lipschitz  condition.
Under the conditions, we provide existence and uniqueness results of Backward Stochastic Dynamic Equations on time scales.

The paper is organized as follows. In section 2, we provide basic definitions regarding time scales.
In section 3, we give some stochastic notations and results on time scales, we state and prove the martingale decomposition theorem on time scales.
And finally, in section 4, we give the main results: the existence and uniqueness solution of backward stochastic dynamic equations on time scales.

\section{Preliminaries}

For the theory of calculus on time scales, we refer to the original work by Higer \cite{Higer1990}.
In this section, we give basic definitions concerning the time scales and present some applications.
For more additional information on the theory of time scales, we refer the reader to the monograph by Bohner and Peterson \cite{advance}.

A \emph{time scale} $\mathds{T}$ is a nonempty closed subset of the real numbers $\mathbf{R}$.
The distance between the points $t,s\in \mathds{T}$ is defined as the normal distance on $\mathbf{R}:|t-s|$.
For t$\in \mathds{T}$, we define the \emph{forward jump operator} $\sigma$: $\mathds{T} \rightarrow \mathds{T}$ by $\sigma(t)=\inf\{s\in \mathds{T}: s>t\}$,
while the \emph{backward jump operator} $\rho$: $\mathds{T} \rightarrow \mathds{T}$ by $\rho(t)=\sup\{s\in \mathds{T}: s<t\}$.
Of course both $\sigma(t)$ and $\rho(t)$ are in $\mathds{T}$ when $t \in \mathds{T}$.
This is because we assume that $\mathds{T}$ is a closed subset of $\mathbf{R}$.
We say that $t$ is \emph{right-scattered} (\emph{left-scattered}, \emph{right-dense}, \emph{left-dense}), if $\sigma(t)>t$ ($\rho(t)<t$, $\sigma(t)=t$, $\rho(t)=t$) hold.
$\mu(t)=\sigma(t)-t$ is called \emph{graininess}, $\nu(t)=t-\rho(t)$ is called \emph{backward graininess}.
The set of right-scattered points of a time scale is at most countably infinite.
For $a,b \in \mathds{T}$ with $a\leq b$ we define the closed interval in $\mathds{T}$ by $[a,b]_{\mathds{T}}=[t\in \mathds{T}:a\leq t \leq b]$.
Other types of intervals are defined similarly.
We introduce the set $\mathds{T}_k$
if $\mathds{T}$ has a right-scattered minimum $t_2$, then $\mathds{T}_k=\mathds{T}-\{t_2\}$, otherwise $\mathds{T}_k=\mathds{T}$.
If $t\in \mathds{T}_k$, then the $\nabla$-\emph{derivative} of $f$ at the point $t$ is defined to be the number $f^{\nabla}(t)$(provided it exists) with the property that for each $\epsilon>0$,
there is a neighborhood $U$(in $\mathds{T}$)of $t$ such that $|f(\rho(t))-f(s)-f^{\nabla}(t)[\rho(t)-s]|\leq \epsilon |\rho(t)-s|, \quad \forall s\in U$.
Now suppose that $f:\mathds{T}\rightarrow \mathbf{R}$. Continuity of $f$ is defined in the usual manner.
A function $f$ is called right-dense continuous (rd-continuous) on $\mathds{T}$ iff it is continuous at every right-dense point and the left-sided limit exists at every left-dense point.
It is similar to the notation of left-dense continuous (ld-continuous).
Note that, on the right-scattered points, $f(t)\neq f(t+)$ for continuous functions on time scales.
Denote $\lim_{\sigma(s)\uparrow t}f(s)$ by $f(t-)$ if this limit exists, $\lim_{\rho(s)\downarrow t}f(s)$ by $f(t+)$.
If $t$ is left-scattered then $f(t-)=f(\rho(t))$, right-scattered then $f(t+)=f(\sigma(t))$.

Let $A$ be an increasing right-continuous function of finite variation defined on $\mathds{T}$.
We denote by $\mu^A_{\nabla}$ the Lebesgue $\nabla$-measure associated with $A$.
For any $\mu^A_{\nabla}$-measure function $f:\mathds{T}\rightarrow \mathbf{R}$,
we write $\int_a^tf_s\nabla A_s$ for the integral of $f$ with respect to the measure $\mu_{\nabla}^A$ on $\left(0,t\right]$.
It is seen that the function $t \mapsto \int_a^tf_s\nabla A_s$ is c$\grave{a}$dl$\grave{a}$g.
For details, we can refer to \cite{first}.
\subsection{Stochastic calculus on time scales}

As time scales unify continuous and discrete analysis,
we assume that the readers are familiar with the concepts and basic properties of the stochastic process in both discrete-time and continuous-time.
We denote by $\mathbf{R}^k$ the k-dimensional Euclidean space,
equipped with the standard inner product
$(\cdot ,\cdot)$, and the Euclidean norm $|\cdot|$.
We also denote by $\mathbf{R}^{k\times d}$ the collection of all $k\times d$ real matrices, and for matrix $z=(z_{ij})_{k\times d}$.
We denote $z_i:=(z_{i1},\cdots,z_{id})^T$ and $|z|:=\sqrt{tr(zz^T)}$,
where $z^T$ represents the transpose of $z$.
In our study, we will be concerned with bounded time scales in which $a=\inf \mathds{T}$ and $b=\sup\mathds{T}$ and $a,b\in \mathds{T}$ are both finite.
For simplicity, we consider $a=0, b=T>0$. By convention, we write $\rho(0)=0$.
The basic definitions of process on time scales are analogous to that of classical discrete-time processes or continuous-time processes,
such as $\{\mathcal{F}_t\}_{t\in \mathds{T}}$-filtration,
$\mathcal{M}_{\mathds{T}}$-the space of uniformly integrable martingales,
$\mathcal{M}_{\mathds{T}}^2$-the space of square integrable martingales with $E\{\sup |M|_t^2\}< \infty$.
We add the subscript $\mathds{T}$ to spaces on time scales.
Most results of martingales also hold in the time scale cases.
We are interested in the distinctions between them.
We assume that we are working on a probability space $(\Omega,\mathcal{F},\mathbf{F}=(\mathcal{F}_t)_{t\in \mathds{T}},\mathbf{P})$
with the filtration $\{\mathcal{F}_t\}_{t\in \mathds{T}}$ satisfying the usual conditions
($\mathcal{F}_t\neq\mathcal{F}_{t+}$ on right-scattered points even it is continuous on time scales, since right-continuous has no meaning at right-scattered points).
For simplicity, $\mathcal{F}_{T}=\bigvee_{t\in\mathds{T}}\mathcal{F}_t$.

We denote by
$\mathcal{M}^2_{\mathds{T}}$ the space of square integrable $\{\mathcal{F}_t\}_{t\in\mathds{T}}$-martingales on time scales and consider $M \in \mathcal{M}^2_{\mathds{T}}$.
Since $M^2$ is a submartingale, following the Doob-Meryer decomposition theorem on time scales \cite{first}, there exists uniquely a natural increasing process
$(\langle M\rangle_t)_{t\in\mathds{T}}$, such that $M^2-\langle M\rangle$ is an $\mathbf{F}$-martingale.
The natural increasing process $\langle M\rangle$ is called characteristic of the martingale $M$.

We define the quadratic co-variation $[M,N]$ of two processes $M,N$ similar to \cite{first}, Definition 3.13.
If $N=M$ we write $[M]_t$ for $[M,M]_t$ and call it the quadratic variation of $M$. For partitions $\{t_i\}$ of $[s,t]_{\mathds{T}}$ with $\max_i(\rho(t_i)-t_{i-1})\leq 2^{-n}$, we have
$$E[(M_t-M_s)^2|\mathcal{F}_s]=E[\sum_{i}(M_{t_i}-M_{t_{i-1}})^2|\mathcal{F}_s]\rightarrow E[[M]_t-[M]_s|\mathcal{F}_s].$$
It means, $M^2-[M]$ is an $\mathbf{F}$-martingale, which implies that $[M]-\langle M \rangle$ is also a $\mathbf{F}$-martingale.

As the one-dimensional case was given in \cite{brownian}, the multi-dimension Brownian motion can be constructed as the classical product space method \cite{GTM}.
Now we give a result about multi-dimension Brownian motion on time scales.
\begin{lemma}
Let $\{W_t=(W_t^1.W_t^2,\cdots,W_t^d),\mathcal{F}_t;t\in\mathds{T}\}$ be a d-dimensional Brownian motion.
The processes
$$M_t^i=W_t^i-W_0^i,~\mathcal{F}_t;~t\in\mathds{T},~1\leq i\leq d,$$
are continuous, square-integrable martingales, with $\langle M^i,M^j\rangle=t\delta_{ij};1\leq i,j\leq d$.
$[M^i,M^j]=\delta_{ij}\{\lambda([t_0,t]\cap \mathds{T})+\sum_{t_0\leq a_n< b_n\leq t}(W_{b_n}-W_{a_n})^2\}, \forall t\in \mathds{T}$,
where $\lambda$ denotes the classical Lebesgue measure
and $\cup_{n=1}^{\infty}(a_n,b_n)=\left[t_0,\infty\right)\setminus \mathds{T}$ is the expression
for the open subset $\left[t_0,\infty\right)\setminus \mathds{T}$ of $\mathbf{R}$ as the countable of disjoint open intervals\cite{quadratic}.
Furthermore, the vector of martingales $M=(M^1,\cdots,M^d)$ is independent of $\mathcal{F}_0$.
\end{lemma}
\begin{remark}
The Levy martingale characterization of Brownian motion failed on time scales, that's a continuous martingale with $\langle X^i,X^j\rangle=t\delta_{ij}$, cannot be a Brownian motion.
\end{remark}
Our stochastic integral on time scales is based on Nguyen Huu Du \& Nguyen Thanh Dieu \cite{first},
in which they established the stochastic $\nabla$-integral w.r.t. square integral martingales, and extended to special semimartingales \cite{SDE2013}.
Consider the integral with respect to Brownian motion $W$ on time scales (square integral on time scales):
let $\mathcal{L}_{\mathds{T}}^2(\left(0,T\right];W)$  be the space of all real-valued, predictable processes $\phi=\{\phi_t\}_{t\in \mathds{T}}$ satisfying
$\|\phi\|^2_{t,W}=E\int_0^t|\phi_s|^2\nabla \langle W\rangle_s=E\int_0^t|\phi_s|^2\nabla s <\infty$.
By using [\cite{first}, Definition 3.6], we can define the integral:
$I_t(X)=\int_0^tX_s \nabla W_s, \quad X\in \mathcal{L}_{\mathds{T}}^2(\left(0,T\right];W)$.
The space $\mathcal{L}_{\mathds{T}}^2(\left(0,T\right];W)$
is actually the $L^2$ space under the measure given by $\nabla \langle W\rangle \times dP$.
Now we define
\textbf{the Multi-dimension Stochastic Integral:}
\begin{definition}
Let $\{W_t,t\in\mathds{T}\}$ be a $d$-dimension Brownian motion on time scales,
$\{X_t,t\in\mathds{T}\}$ is $k\times d$ matrix process, $\forall 1\leq i\leq n,1\leq j\leq d$, $X^{i,j}\in\mathcal{L}^2_{\mathds{T}}(\left(0,T\right];W^j)$.
Define:
\begin{equation}
I_{t,n}(X)=\int_0^tX_s\nabla W_s=\int_0^t
\left(
\begin{array}{lll}
X_{11}& \cdots& X_{1d}\\
\vdots&  \quad &\vdots\\
X_{n1}& \cdots& X_{nd}\\
\end{array}
\right)
\left(
\begin{array}{l}
\nabla W_s^1\\
\vdots\\
\nabla W_s^d\\
\end{array}
\right).
\end{equation}
We denote the multi-dimensional stochastic integral $I_{t,k}(X)$ for $X$,
where $X$ is $(\mathcal{F}_t)_{t\in\mathds{T}}$-predictable, $\mathbf{R}^{k\times d}$-valued processes
such that $E\int_0^T|X_s|^2\nabla t<\infty$, denote the set by $\mathcal{L}^2_{\mathds{T}}((\left(0,T\right];\mathbf{R}^{k\times d})$
(short for $\mathcal{L}^2_{\mathds{T},k\times d}$).
We denote the integral also by $I_t(X)$ for short.
The i-th component of $I_t(X)$ is:
$$I_t^i(X)=\sum_{j=1}^{d}\int_0^t X_s^{ij}\nabla W_s^j, \quad 1\leq i\leq k, \quad t\in \mathds{T}.$$
\end{definition}
Clearly that $I_{t}(X)$ belongs to $\mathcal{M}^2_{\mathds{T},k}$, the $\mathbf{R}^k$-valued square-integral martingale on time scales.
Therefore we also introduce
\begin{itemize}
\item $\mathcal{M}_{\mathds{T},k}^{2,c}$-the continuous $\mathbf{R}^k$-valued square-integral martingale processes on time scale with $M_0=0 \quad\mathbf{P}-a.s.$,
\item $\mathcal{M}_{\mathds{T},k}^{2,*}$-the subset of $\mathcal{M}_{\mathds{T},k}^{2,c}$, such that for each $M\in\mathcal{M}_{\mathds{T},k}^{2,*}$,
    there exists $X\in\mathcal{L}_{\mathds{T},k\times d}^2$, and $M_t=I_t(X)$
\end{itemize}
for any integer $k$.
We will study properties of these spaces next section.
For more properties about the stochastic integral, readers could see \cite{first}.
\begin{remark}
Suman \cite{existence} also defined a stochastic integral on time scale, the $\Delta$-integral.
For the stochastic integral, the integrands are actually predictable in order to establish It\^{o} type integral.
On this occasion, the stochastic integral has martingale property.
(For the discrete case, the martingale transform is the discrete-time version of the stochastic integral with respect to a martingale.)
\end{remark}

\section{Martingale decomposition theorem on time scales}
The fundamental tool on time scales analysis is the countable dense subset.
The countable dense subset will play a role analogous to the dyadic rational numbers that played in the classical analysis from discrete-time to continuous-time.
For any $\delta>0$, consider a partition of $[0,T]_{\mathds{T}}$ inductively by letting $t_0=0$ and for $i=1,2,\cdots$, set:
\begin{equation}
t_i=
\left\{
\begin{array}{lll}
\sup B_i& if~  B_i\neq \emptyset\\
\sigma(t_{i-1})& if~  B_i=\emptyset\\
\end{array}
\right.
\quad where \quad B_i=(t_{i-1},t_{i-1}+\delta]_{\mathds{T}} \bigcap [0,T]_{\mathds{T}}.
\end{equation}
The partition is given in \cite{integration}, \cite{first}.
On time scales, the interval $|t_{i+1}-t_i|$ will not converge to zero.
A more specific result was given by David Grow \cite{brownian}.
Now we provide the optional sampling theorem on time scales to show the basic analysis method on general time scales.

\begin{lemma}[Optional Sampling Theorem on Time Scales] \label{OptionalSampling}
If $X$ is a right-continuous martingale (submartingale) on bounded time scales $\mathds{T}$ with a last element $X_T$
and $S_1,S_2$ are two bounded stopping times with $S_1\leq S_2$ on $\mathds{T}$, then
\begin{equation}\nonumber
E[X_{S_2}|\mathcal{F}_{S_1}]=X_{S_1}(\geq X_{S_1})\quad  \mathbf{P}-a.s..
\end{equation}
\end{lemma}
\begin{proof}
Let $[0,T]_{\mathds{T}}$ be a time scale, and let $\Pi_n=\{t_0,t_1,\cdots,t_n\}\subseteq [0,T]_{\mathds{T}}$ be a partition
(as the same in \cite{integration}) of $\mathds{T}$, where $0=t_0<t_1<t_2<\cdots<t_n=T$.
Consider the sequence of random times
\begin{equation}\label{ppp}
S_1^n(\omega;P_n)=\rho(t_{i+1})\quad if\quad  t_i\leq S_1(\omega)<t_{i+1},
\end{equation}
and the similarly defined sequences $\{S_2^n\}$.
These are stopping times.
For every fixed integer $n\geq 1$, both $S_1^n$ and $S_2^n$ take on a countable number of values and we also have $S_1^n\leq S_2^n$.
Therefore, by the discrete optional sampling theorem we have $\int_A X_{S_1^n}dP\leq \int_A X_{S_2^n}dP$ for every $A\in \mathcal{F}_{S_1^n}$.
$S_1\leq S_1^n$ implies $\mathcal{F}_{S_1}\subset \mathcal{F}_{S_1^n}$, the preceding inequality also holds for every $A\in\mathcal{F}_{S_1}$.

The discrete martingale results show that the sequence of random variables $\{X_{S_1^n}\}$ is uniformly integrable, and the same is of course true for $\{X_{S_2^n}\}$.
$X_{S_2}=\lim_{n\rightarrow\infty}X_{S_2^n}(\omega)$ and $X_{S_1}=\lim_{n\rightarrow\infty}X_{S_1^n}(\omega)$ hold for a.e. $\omega \in\Omega$.
It follows from uniform integrability that $X_{S_1}$, $X_{S_2}$ are integrable, and that $\int_AX_{S_1}dP\leq \int_AX_{S_2}dP$ holds for every $A\in \mathcal{F}_{S_1}$.
\end{proof}

Via the identity $M_t=E[M_T|\mathcal{F}_t]$, each $M\in \mathcal{M}_{\mathds{T}}^{2}$ can be identified with its terminal value $M_T\in L^{2}(\Omega,\mathcal{F}_T,\mathbf{P})$
(in general the terminal variables can be extended $M_{\infty}$ if exists).
$\mathcal{M}_{\mathds{T}}^{2}$ becomes a Hilbert space isomorphic to $L^{2}(\Omega,\mathcal{F}_T,\mathbf{P})$ , if endowed with the inner product
$$(M,N)_{\mathcal{M}^2_{\mathds{T}}}:=E[M_TN_T],\quad \|M\|_{\mathcal{M}^2_{\mathds{T}}}=\|M_T\|_{L^2}.$$
Indeed, if $(M^n)$ is a Cauchy sequence for $\|\cdot\|_{\mathcal{M}^2_{\mathds{T}}}$,
then the sequence $(M^n_T)$ is Cauchy in $L^2(\Omega,\mathcal{F}_{T},\mathbf{P})$ and so goes to a limit $M_T$ in this space;
then if $M$ is the martingale with terminal variable $M_T$, it belongs to $\mathcal{M}^2_{\mathds{T}}$ and $\|M^n-M\|\rightarrow 0$.
The set of all continuous elements of on time scales $\mathcal{M}^2_{\mathds{T}}$ denoted by $\mathcal{M}^{2,c}_{\mathds{T}}$, is a closed subspace of
the Hilbert space $\mathcal{M}^2_{\mathds{T}}$.

In this part, we study some properties of the spaces $\mathcal{M}^{2}_{\mathds{T},k}$ on time scales.
We may define a measure $\mu_p$ on
$(\mathds{T}\times \Omega,\mathcal{B}(\mathds{T})\bigotimes\mathcal{F}_{\rho(t)})$ by
$$\mu_p(A)=E(\int_{\mathds{T}}1_A(s,\omega)\nabla \langle W\rangle_s).$$

We recall $\mathcal{L}_{\mathds{T},k\times d}^2$
the space of all $\mathbf{R}^{k\times d}$-valued, predictable processes $\phi=\{\phi_t\}_{t\in \mathds{T}}$ satisfying
$$\|\phi\|^2_{t,W}=E\int_0^t|\phi_s|^2\nabla s <\infty.$$
Similar to Lemma 2.2 chapter 3 in Karatzas's book \cite{GTM}, we have
\begin{lemma}
The $\mathcal{L}^2_{\mathds{T},k\times d}$ space is a closed space with the norm $\|\cdot\|_{T,W}$.
\end{lemma}
\begin{proof}
We define a Hilbert space
$$\mathcal{H}_T=L^2(\mathds{T}\times \Omega,\mathcal{B}(\mathds{T})\otimes\mathcal{F}_{\rho(T)},\mu_p)$$
Obviously, $\mathcal{L}^2_{\mathds{T},k\times d}$ is a subspace of $\mathcal{H}_T$.
Now we prove that it is closed.
Suppose that $\{X^m\}_{m=1}^{\infty}$ is a convergent sequence in $\mathcal{L}^2_{\mathds{T},k\times d}$
with limit $X\in \mathcal{H}_T$.
Thus the sequence has a convergent subsequence which converges almost surely under $\mu_p$, also denoted by $\{X^m\}_{m=1}^{\infty}$.
Therefore
$$\mu_p\{(t,\omega):\lim_{n\rightarrow \infty}X_t^m(\omega)\neq X_t(\omega);\quad t\in \mathds{T}\}=0,$$
$X\in\mathcal{H}_T$, thus $X$ is $\mathcal{B}(\mathds{T})\otimes\mathcal{F}_{\rho(T)}$-measurable.

Restricted on $[0,t]_{\mathds{T}}$ for $0<t\leq T$, repeating the above procedure,
by the uniqueness of convergence, we can get that $X$ is $\mathcal{B}([0,t]_{\mathds{T}})\otimes\mathcal{F}_{\rho(t)}$-measurable.
Therefore $X$ is predictable and belongs to $\mathcal{L}^2_{\mathds{T},k\times d}$.
The proof is complete.
\end{proof}

We can define an inner product on $\mathcal{L}_{\mathds{T},k\times d}^2$ by
$(X,Y)_1=E\int_0^tX_{s}Y_{s}\nabla s$.
Recall the inner product on $L^{2}(\Omega,\mathcal{F}_T,\mathbf{P};\mathbf{R}^k)$ (simply denoted by $L^2_k(\mathcal{F}_T)$), $(I_1,I_2)_2=E[I_1I_2]$.
Now we consider the mapping $X\mapsto I_T(X)$ from $\mathcal{L}_{\mathds{T},k\times d}^2$ to $L^{2}_k(\mathcal{F}_T)$.
This mapping preserves inner products:
$$(X,Y)_1=E\int_0^TX_tY_t \nabla t=E[I_T(X)I_T(Y)]=(I_T(X),I_T(Y))_2.$$
Denote $\mathcal{R}_k(\mathcal{F}_T)\triangleq \{I_T(X);X\in\mathcal{L}_{\mathds{T},k\times d}^2\}$.
Since $\mathcal{L}_{\mathds{T},k\times d}^2$ is closed, any convergent sequence in $\mathcal{R}_k(\mathcal{F}_T)$
is also Cauchy, its preimage sequence in $\mathcal{L}_{\mathds{T},k\times d}^2$ must have a limit in $\mathcal{L}_{\mathds{T},k\times d}^2$.
It follows that $\mathcal{R}_k(\mathcal{F}_T)$ is closed in $L^2_k(\mathcal{F}_T)$.
Let us denote by $\mathcal{M}^{2,*}_{\mathds{T}}$ the subset of $\mathcal{M}^{2,c}_{\mathds{T},k}$ which consists of stochastic integrals
$$I_t(X)=\int_0^tX_s\nabla W_s; \quad 0\leq t\leq T,\quad t\in \mathds{T},$$
of processes $X\in\mathcal{L}_{\mathds{T},k\times d}^2$:
\begin{equation}
\mathcal{M}^{2,*}_{\mathds{T},k}\triangleq\{I(X);X\in\mathcal{L}_{\mathds{T},k\times d}^2\}\subseteq \mathcal{M}^{2,c}_{\mathds{T},k}\subseteq \mathcal{M}^{2}_{\mathds{T},k}.
\end{equation}

Now we state the "fundamental decomposition theorem" for the martingales w.r.t Brownian motion on time scales.

\begin{theorem}\label{multidecomtheorem}
For every $M\in \mathcal{M}_{\mathds{T},k}^{2}$, with $M_0=0,~\mathbf{P}-a.s.$, we have the decomposition
$$M_t=I_t(X)+N_t, \quad \forall t\in\mathds{T},$$
where $X\in\mathcal{L}^2_{\mathds{T},k\times d}$, $I_t(X)\in \mathcal{M}_{\mathds{T},k}^{2,*},~N\in \mathcal{M}_{\mathds{T},k}^{2}$ with $N_0=0$
and $N$ is orthogonal to every element of $\mathcal{M}_{\mathds{T},k}^{2,*}$.
\end{theorem}
\begin{proof}
\textbf{Uniqueness}:
Suppose that there exists a process $Y\in \mathcal{L}_{\mathds{T},k\times d}^2$ such that $M_t=I_t(Y)+N_t$, where $N\in \mathcal{M}_{\mathds{T},k}^{2}$ has the property
$$\langle I(X),N\rangle_t=0,\quad \forall X\in \mathcal{L}_{\mathds{T},k\times d}^2.$$
Such a decomposition is unique (up to indistinguishability);
indeed, if we have $M=I(Y')+N'=I(Y'')+N''$ with $Y',Y''\in \mathcal{L}_{\mathds{T},k\times d}^2$ and both $N'$ and $N''$ satisfy the property,
then
$$Z\triangleq N''-N'=I(Y'-Y'')$$
is in $\mathcal{M}_{\mathds{T},k}^{2}$ with $Z_0=0$ and $\langle Z\rangle=\langle Z,I(Y'-Y'')\rangle=0$.
Then $P[Z_t=0,\forall t\in \mathds{T}]=1$ from the Optional Sampling Theorem.
(Similar results as Problem1.5.12 in Karatzas\cite{GTM}, applying optional sampling theorem to the martingale $Z^2-\langle Z\rangle$.)
The decomposition is unique up to indistinguishability.

\textbf{Existence}:

Since $\mathcal{R}_k(\mathcal{F}_T)$ is a closed subspace of $L^2_k(\mathcal{F}_T)$,
we can denote its orthogonal complement by $\mathcal{R}_k^{\perp}(\mathcal{F}_T)$.
The random variable $M_T$ is in $L^2_k(\mathcal{F}_T)$, so it admits the decomposition
\begin{equation}\label{de}
M_T=I_T(Y)+N_T,
\end{equation}
where $Y\in\mathcal{L}_{\mathds{T},k\times d}^2$ and $N_T\in L^2_k(\mathcal{F}_T)$ satisfies
$$E\{N_TI_T(X)\}=0;\quad \forall X\in\mathcal{L}_{\mathds{T},k\times d}^2.$$
We construct a martingale through $N_T$ by
$N_t=E(N_T|\mathcal{F}_t)$.
Obviously $N\in\mathcal{M}_{\mathds{T}}^2$.
Taking conditional expectation under $\mathcal{F}_t$ on $M_T$, we obtain
$$M_t=I_t(Y)+N_t.$$
It remains to show that $N$ is orthogonal to every square-integrable martingale of the form
$\{I(X);X\in\mathcal{L}_{\mathds{T},k\times d}^2\}$,
or equivalently, that $\{N_tI_t(X),\mathcal{F}_t\}_{t\in\mathds{T}}$ is a martingale.
Note that each martingale has a right continuous modification.
So now we suppose that $N$ is right continuous.

According to lemma \ref{OptionalSampling},
we only need to prove that $E[N_SI_S(X)]=E[N_0I_0(X)]=0$ holds for every stoping time $S$ of the filtration $\{\mathcal{F}_t\}_{t\in \mathds{T}}$, with $S<T$ (since $I_0(X)=0$).
The integral has $I_S(X)=I_{T}(\tilde{X})$, where $\tilde{X}_t(\omega)=X_t(\omega)1_{\{t\leq S(\omega)\}}$
is a process in $\mathcal{L}_{\mathds{T},k\times d}^2$.
Therefore, by the Optional Sampling Theorem, we have
$$E[N_SI_S(X)]=E[E(N_T|\mathcal{F}_{S})I_S(X)]=E[N_TI_T(\tilde{X})]=0.$$
The proof is complete.
\end{proof}

\begin{remark}
\begin{enumerate}
\item Note $f(t)\neq f(t-)$ at left-scattered points for continuous function on time scales, $
N\in \mathcal{M}_{\mathds{T},k}^{2,c}$ is continuous, but $\nabla N_t\neq 0$ at left-scattered points.
\item Let $\mathds{T}=\mathbf{R},k=1$ (Karatzas \cite{GTM}, Proposition 4.14 one-dimension decomposition)
$\mathcal{M}_{\mathbf{R}}^{2,c}$ and $\mathbf{M}_{\mathbf{R}}^{2,*}$ are actually coincide, the component $N_t$ in the decomposition is actually $\nabla N_t=0$.
The predictable process space is isomorphic to the adapted process space \cite{GTM2}.
\item For Brownian motion on general time scales, even on the augmentation filtration of the filtration generated by $W$,
$\mathcal{M}_{\mathds{T},k}^{2,c}$ and $\mathcal{M}_{\mathds{T},k}^{2,*}$ do not coincide, see the following example.
\item Let $\mathds{T}=\mathds{N}$ (Follmer, Hans, \cite{discreteKW}Theorem 10.18),
 it is the discrete time version of the Kunita-Watanabe decomposition w.r.t a sequence of normal distribution random variables.
\end{enumerate}
\end{remark}
\begin{example}
For a martingale $B_t^2-t$ on time scales $\mathds{T}=[0,1]\cup\{3,4,5\}$, we have
$$B_t^2-t=(B_t-B_s)^2-(t-s)+2B_s(B_t-B_s)+B_s^2-s,$$
so for $(s,t)_{\mathds{T}}\bigcap \mathbf{R}=\emptyset$,
\begin{itemize}
\item
$N_t=(B_t-B_s)^2-(t-s)$ is the orthogonal martingale part w.r.t. Brownian interal.
\item  $I_t=2B_s(B_t-B_s)$ is the It\^{o} integral part w.r.t Brownian motion at left-scattered point. In this case $N_t$ is orthogonal to $I_t$.
\end{itemize}
\end{example}


\section{BS$\nabla$E driven by Brownian Motion on time scales}

Denote $(\Omega,\mathcal{F},\mathbf{P})$ be a probability space equipped with a complete filtration $\mathcal{F}=(\mathcal{F}_t)_{t\in\mathds{T}}$
generated by a d-dimensional Brownian Motion $W_t$ on time scales, and augmented by all the $\mathbf{P}$-null sets in $\mathcal{F}$.
$E[X_t|\mathcal{F}_s]$ or $E^{\mathcal{F}_s}$ denotes the conditional expectation with respect to the filtration $\mathcal{F}_s$.

For simplicity, we consider the general BSDE on time scales
\begin{equation}\label{bsde}
\left\{
\begin{array}{ll}
Y_t=Y_T+\int_t^Tg(u,Y_{u-},Z_u)\nabla u-\int_t^TZ_u\nabla W_u-(N_T-N_t),\\
Y_T=\xi.\\
\end{array}
\right.
\end{equation}
We call $g$ the driver of the BSDE on time scales and the pair $(\xi,g)$ the data of the BSDE on time scales.
\begin{itemize}
\item $L^2_k(\mathcal{F}_t)$,
 short for $L^2(\Omega,\mathcal{F}_t,\mathbf{P};R^k)$,
 the space of $\mathbf{R}^{k}$-valued random vectors $\xi$ that are $\mathcal{F}_t$-measurable and satisfy $E[\|\xi\|^2]<\infty$,
\item $\mathcal{S}^2_{\mathds{T}}$, short for $\mathcal{S}^2_{\mathds{T}}(0,T;\mathbf{R}^{k})$,
     the set of
    $\mathbf{R}^{k}$-valued, adapted and continuous processes $(\phi_t)_{t\in [0,T]_{\mathds{T}}}$ such that $E[sup_{t\in[0,T]_{\mathds{T}}}|\phi_t|^2]<\infty,$
\item $M^2_{\mathds{T}}$, short for $M^2_{\mathds{T}}(0,T;\mathbf{R}^{k\times d})$,
the set of $\mathbf{R}^{k\times d}$-valued, $\mathcal{F}_t$-progressively measurable, predictable processes $(\phi_t)_{t\in[0,T]_{\mathds{T}}}$
      such that
    $$E\{\int_0^T|\phi_t|^2\nabla t\}<\infty,$$
\item $\mathcal{M}_{\mathds{T},k}^{2,\perp}$-subset of $\mathcal{M}_{\mathds{T},k}^{2,c}$,
the set of all the $k$-dimensional martingales, such that each martingale is orthogonal to that in $\mathcal{M}_{\mathds{T},k}^{2,*}$,
\end{itemize}
for any integer $k$.

Obviously, $\mathcal{S}^2_{\mathds{T}}$ is a Banach space and
$M^2_{\mathds{T}}$
is a Hilbert space.
A solution to a BSDE on time scales is a triple of process $(Y, Z, N)$ satisfying the above equation, such that $Y$ is $\mathbf{R}^k$-valued,
continuous and adapted, $Z$ is $\mathbf{R}^{k\times d}$-valued and predictable, $N$ is a martingale process strongly orthogonal to $W$.
For terminal condition $\xi$, and generator $g$, we make the following assumptions:

\textbf{Assumption(H):}
\begin{itemize}
\item (H1) $g$ is defined as $\Omega\times \mathds{T}\times \mathbf{R}^k\times\mathbf{R}^{k\times d}\rightarrow \mathbf{R}^k$, such that
$\forall(y,z)\in \mathbf{R}^k\times  \mathbf{R}^{k\times d}$, $g(\cdot,y,z)$ is $\mathcal{F}_t$ progressively measurable.
The deterministic integral part could also be $g(\omega,t,Y_{t},Z_t)$ or $g(\omega,t-,Y_{t-},Z_{t-})$.
For simplicity, we only consider the $g(\cdot,Y_{t-},Z_t)$ case, which corresponding to the implicit discrete backward difference equation (\ref{discretebsde1}).
\item (H2)
$$
 \int_0^T|g(\cdot ,0,0)|^2\nabla t\in L^2_k(\mathcal{F}_T).
$$
\item (H3)
we also assume $g$ satisfies Lipschitz condition w.r.t. $(y,z)$:
there exists a constant $L>0$, $\forall y,y'\in \mathbf{R}^k$, $z,z'\in \mathbf{R}^{k\times d}$
$$
 |g(t,y,z)-g(t,y',z')|\leq L(|y-y'|+|z-z'|).
$$
\item  (H4) $\xi \in L^2_k(\mathcal{F}_T)$
\end{itemize}
for any inter $k$.

\begin{theorem}\textbf{(Main Result)}\label{theorem}
Assuming (H) is satisfied,
BS$\nabla$E(\ref{bsde}) admits a unique triple of solution
$(Y, Z, N)\in
\mathcal{S}^2_{\mathds{T}}\times M^2_{\mathds{T}}\times\mathcal{M}_{\mathds{T},k}^{2,\perp}$.
\end{theorem}
\begin{remark}
$\mathds{T}=\mathbf{R}$, Theorem \ref{theorem} degenerates to the classical BSDE.
\end{remark}
For simplicity, we first consider $g$ real-valued and not contain $(y,z)$.
\begin{lemma}\label{bsdelemma}
For a fixed $\xi \in L^2_k(\mathcal{F}_T)$ and $g_0({\cdot}):\Omega\times \mathds{T}\rightarrow \mathbf{R}^k$ is a progressively measurable
and $\mathcal{F}_t$-predictable process which satisfies
$$E\int_0^T|g_0(s)|^2\nabla s<\infty.$$
Then for BSDE(\ref{bsde}) with parameters $\xi$ and $g_0$,
there exists a triple solution $(Y,Z,N)\in \mathcal{S}^2_{\mathds{T}}\times M^2_{\mathds{T}}\times \mathcal{M}_{\mathds{T},k}^{2,\perp}$.
Moreover,
\begin{equation}\label{bsdeinequality}
\begin{split}
&|Y_t|^2+E^{\mathcal{F}_t}\int_t^T\{\frac{\beta}{2} |Y_{s-}|^2+|Z_s|^2\}e_{\beta}(s,t)\nabla s+\int_t^Te_{\beta}(s,t)\nabla[N_s ]\\
\leq & E^{\mathcal{F}_t}|\xi|^2e_{\beta}(T,t)+\frac{2}{\beta}E^{\mathcal{F}_t}\int_t^T|g_0(s)|^2e_{\beta}(s,t)\nabla s.\\
\end{split}
\end{equation}
$\beta$ is a positive constant.
\end{lemma}

\begin{proof}
\textbf{Existence}:

We denote
\begin{equation}\nonumber
M_t=E^{\mathcal{F}_t}[\xi +\int_0^Tg_0(s)\nabla s].
\end{equation}
Apparently, $M_t$ is a square integrable martingale.
Following the Martingale Decomposition Theorem on time scales,
there exists a unique predictable process $Z_t\in\mathcal{L}^2_{\mathds{T},k\times d}$ and $N\in\mathcal{M}_{\mathds{T},k}^{2,\perp}$,
s.t. $M_t=M_0+\int_0^tZ_s\nabla W_s+N_t$, $M_T=M_0+\int_{\left(0,T\right]}Z_s\nabla W_s+N_T$.
Then $M_t=M_T-\int_t^TZ_s \nabla W_s-(N_T-N_t)$. As also $M_T=\xi+\int_0^T g_0(s)\nabla s$,
denote $Y_t=M_t-\int_0^tg_0(s)\nabla s=M_T+\int_t^T g_0(s)\nabla s-\int_t^T Z_s\nabla W_s-(N_T-N_t)$.
It can be seen that $Y\in\mathcal{S}^2_{\mathds{T}}$, we get the existence.

\textbf{Uniqueness:}

For any fixed $\beta$, with $1+\beta \mu(t)>0$, the $\Delta$-exponential function $e_{\beta}(t,t_0)$, defined by \cite{advance},
is the solution $y(\cdot)$ of the initial value problem
\begin{equation}\label{expfunction}
\nabla y(t)=\beta y(t-), \forall t>t_0, t\in \mathds{T}, y(t_0)=1.
\end{equation}
Applying Ito's formula to $|Y|_s^2e_{\beta}(s,t)$ on $[t,T]_{\mathds{T}}$, where
\begin{equation}\nonumber
Y_t=\xi+\int_t^Tg_0(s)\nabla s-\int_t^TZ_s\nabla W_s-N_T+N_t,
\end{equation}
or the differential form:
\begin{equation}\nonumber
\nabla Y_t=-g_0(t)\nabla t+Z_t\nabla W_t+\nabla N_t.
\end{equation}
We have
\begin{equation}\nonumber
\begin{split}
|Y_T|^2e_{\beta}(T,t)&=|Y_t|^2+\int_t^T\beta e_{\beta}(s,t)Y_{s-}^2\nabla s+2\int_t^Te_{\beta}(s,t)Y_{s-}\nabla Y_s\\
 &+\int_t^Te_{\beta}(s,t)\nabla [Y]_s+ S\\
&=|Y_t|^2+\int_t^T\beta e_{\beta}(s,t)Y_{s-}^2\nabla s-2\int_t^T e_{\beta}(s,t)Y_{s-} g_0(s)\nabla s\\
& +2\int_t^T e_{\beta}(s,t)Y_{s-} Z_{s}\nabla W_s+2\int_t^T e_{\beta}(s,t)Y_{s-}\nabla N_s\\
& +\int_t^T e_{\beta}(s,t)Z_s^2\nabla[W]_s+\int_t^T e_{\beta}(s,t)\nabla[N]_s+S,\\
\end{split}
\end{equation}
where
\begin{equation}\nonumber
\begin{split}
S=&\sum_{s\in \left(t,T\right]}\{|Y_s|^2e_{\beta}(s,t)-|Y_{s-}|^2e_{\beta}(s,t)-2|Y_{s-}|e_{\beta}(s,t)(|Y_s|-|Y_{s-}|)\\
&-e_{\beta}(s,t)(|Y_s|-|Y_{s-}|)^2\}\\
=&0.\\
\end{split}
\end{equation}
Thus we have
\begin{equation}\nonumber
\begin{split}
&|Y_t|^2+\int_t^T\beta e_{\beta}(s,t)Y_{s-}^2\nabla s+\int_t^T e_{\beta}(s,t)Z_s^2\nabla[W]_s+\int_t^Te_{\beta}(s,t)\nabla[N_s]\\
=&|\xi|^2e_{\beta}(T,t)+2\int_t^Te_{\beta}(s,t)Y_{s-} g_0(s)\nabla s-2\int_t^Te_{\beta}(s,t)Y_{s-} Z_s\nabla W_s\\
&-2\int_t^Te_{\beta}(s,t)Y_{s-}\nabla N_s.\\
\end{split}
\end{equation}
Since $g_0(s)$ is $\mathcal{F}_s$-predictable, the above integrals with respect to $\nabla W$ and $\nabla N$ all belong to $\mathcal{M}_{\mathds{T},k}^2$.
By taking the expectation with respect to $\mathcal{F}_t$, we get
\begin{equation}
\begin{split}
&|Y_t|^2+E^{\mathcal{F}_t}\int_t^T[\beta |Y_{s-}|^2+|Z_s|^2]e_{\beta}(s,t)\nabla s+E^{\mathcal{F}_t}\int_t^Te_{\beta}(s,t)\nabla[N]_s\\
=&E^{\mathcal{F}_t}|\xi|^2e^{\beta (T-t)}+E^{\mathcal{F}_t}\int_t^T2Y_{s-}g_0(s)e_{\beta}(s,t)\nabla s\\
\leq &E^{\mathcal{F}_t}|\xi|^2e^{\beta (T-t)}+E^{\mathcal{F}_t}\int_t^T[\frac{\beta}{2}|Y_{s-}|^2+\frac{2}{\beta}|g_0(s)|^2]e_{\beta}(s,t)\nabla s\\
\end{split}
\end{equation}
So we get the inequality.

Now we suppose there are two solutions $(Y^1, Z^1, N^1)$ and $(Y^2, Z^2, N^2)$.
Define $\delta Y_t=Y_t^1-Y_t^2$, $\delta Z_t=Z_t^1-Z_t^2$, $\delta N_t=N_t^1-N_t^2$,
then
\begin{equation}\nonumber
\left\{
\begin{array}{ll}
\delta Y_t=-\int_t^T\delta Z_s \nabla W_s-\delta N_T+\delta N_t,\\
\delta Y_T=0.
\end{array}
\right.
\end{equation}
then $(\delta Y,\delta Z,\delta N)$ satisfy BSDE(\ref{bsde}) with $\xi=0$ and $g=0$.
By the inequality
$$|\delta Y_0|^2+E\int_0^T[\frac{\beta}{2}|\delta Y_{s-}|^2+|\delta Z_s|^2]e_{\beta}(s,0)\nabla s+\int_0^Te_{\beta}(s,0)\nabla[\delta N]_s\leq 0$$
By the continuity of $\delta Y_s,\delta Y_s=0,\delta N_s=C-\delta N_0=0, a.s.-\mathbf{P}$.
The uniqueness proved.
\end{proof}
Now we introduce a new norm on $\mathcal{S}^2_{\mathds{T}}$:
$$\|\phi(\cdot)\|_{\beta,T}=\{E\int_0^T|\phi_{s}|^2e_{\beta} (s,0)\nabla s\}^{\frac{1}{2}}$$
and a new norm on $\mathcal{M}_{\mathds{T},k}^{2,\perp}$: $\|N\|_{\beta,T}=\{E\int_0^Te_{\beta} (s,0)\nabla [N]_s\}^{\frac{1}{2}}$
for any positive integer $k$.
Apparently, for each $\beta>0$, $\|\cdot\|_{\beta,T}$ is equivalent to $\|\cdot\|_{0,T}$ which is the original norm on the corresponding space.
Now we start to prove Theorem \ref{theorem}:
\begin{proof}
For any fixed $(y(\cdot),z(\cdot),n(\cdot))\in \mathcal{S}^2_{\mathds{T}}\times M^2_{\mathds{T}}\times\mathcal{M}_{\mathds{T},k}^{2,\perp}$,
it follows from the lemma \ref{bsdelemma} that it admits a unique triple solution
$(Y(\cdot),Z(\cdot),N(\cdot))\in \mathcal{S}^2_{\mathds{T}}\times M^2_{\mathds{T}}\times\mathcal{M}_{\mathds{T},k}^{2,\perp}$ satisfies
\begin{equation}\nonumber
Y_t=\xi+\int_t^Tg(s,y_s,z_s)\nabla s-\int_t^TZ_s \nabla W_s-N_T+N_t,\quad \forall t\in [0,T]_{\mathds{T}}.
\end{equation}
Hence, we can define an operator
$$(Y.,Z.,N.)=I[(y.,z.,n.)]:\mathcal{S}^2_{\mathds{T}}\times M^2_{\mathds{T}}\times\mathcal{M}_{\mathds{T},k}^{2,\perp}\rightarrow
\mathcal{S}^2_{\mathds{T}}\times M^2_{\mathds{T}}\times\mathcal{M}_{\mathds{T},k}^{2,\perp}.$$
We can prove that $I$ forms a contraction mapping on the Banach space $\mathcal{S}_{\mathds{T}}^2\times M^2_{\mathds{T}}\times\mathcal{M}_{\mathds{T},k}^{2,\perp}$.
Take any $(y^1(\cdot),z^1(\cdot),n^1(\cdot)),(y^{2}(\cdot),z^{2}(\cdot),n^2(\cdot))\in\mathcal{S}_{\mathds{T}}^2\times M^2_{\mathds{T}}\times\mathcal{M}_{\mathds{T},k}^{2,\perp}$,
we denote
$$(Y^1,Z^1,N^1)=I[(y^1,z^1,n^1)],\quad (Y^2,Z^2,N^2)=I[(y^2,z^2,n^2)].$$
and $\delta y=y^1-y^2$, $\delta z=z^1-z^2$, $\delta n=n^1-n^2$.
By equation(\ref{bsdeinequality}), we get
\begin{equation}\nonumber
\begin{split}
&E|\delta Y_0|^2+E\{\int_0^T\{\frac{\beta}{2} |\delta Y_{s-}|^2+|\delta Z_s|^2\}e_{\beta}(s,0)\nabla s+\int_0^Te_{\beta}(s-,0)\nabla[\delta N ]_s\}\\
\leq &\frac{2}{\beta}E\int_0^Te_{\beta}(s,0)|g(s,y^1_{s-},z^1_s)-g(s,y^2_{s-},z^2_s)|^2\nabla s\\
\leq &\frac{2}{\beta}E\int_0^Te_{\beta}(s,0)|L|\delta y_{s-}|+L|\delta z_s||^2\nabla s\\
\leq &\frac{2}{\beta}E\int_0^Te_{\beta}(s,0)[2L^2|\delta y_{s-}|^2+2L^2|\delta z_s|^2]\nabla s.\\
\end{split}
\end{equation}
We get
\begin{equation}
\begin{split}
&E\{\int_0^T\{\frac{\beta}{2} |\delta Y_{s-}|^2+|\delta Z_s|^2\}e_{\beta}(s,0)\nabla s+\int_0^Te_{\beta}(s-,0)\nabla[\delta N ]_s\}\\
\leq &\frac{4L^2}{\beta}E\int_0^Te_{\beta}(s,0)[|\delta y_{s-}|^2+|\delta z_s|^2]\nabla s.
\end{split}
\end{equation}
Let $\beta=8(1+L^2)$,
\begin{equation}\nonumber
\|\delta Y\|_{\beta,T}^2+\|\delta Z\|_{\beta,T}^2+\|\delta N\|_{\beta,T}^2\leq \frac{1}{2}(\|\delta y\|_{\beta,T}^2+\|\delta z\|_{\beta,T}^2).
\end{equation}
We can get that the operator $I$ is contractive.
Hence, in this case, there exists a unique fixed point $(Y,Z,N)\in\mathcal{S}_{\mathds{T}}^2\times M^2_{\mathds{T}}\times\mathcal{M}_{\mathds{T},k}^{2,\perp}$ for the map $I$.
The proof is complete.
\end{proof}
\subsection{The Linear case}
\begin{example}\text{(Nguyen Huu Du 2011\cite{first})}\text{(exponential martingale)}
Let $M$ be a bounded martingale satisfying $\nabla^*M_s= M_s-M_{\rho (s)}\neq-1$ for any $s\in \left(0,T\right]_{\mathds{T}}, s\in\mathds{T}.$ Set
\begin{equation}\nonumber
\tilde{M}_t=M_t-\sum_{s\leq t ,s,t\in \mathds{T}}(M_s-M_{s-}),
\end{equation}
\begin{equation}\nonumber
\mathcal{E}_t=\{\prod_{s\left(0,t\right]}(1+\nabla^*M_s)\}\exp\bigg\{\tilde{M}_t-\frac{1}{2}[\tilde{M}_t]\bigg\}.
\end{equation}
Then, $\mathcal{E}_t$ satisfies Doleans exponential equation
\begin{equation}\nonumber
\mathcal{E}_t=1+\int_{\left(0,t\right]}\mathcal{E}_{s_{-}}\nabla M_s.
\end{equation}
\end{example}
\begin{lemma}
Let $a,b,c$ be predictable bounded processes, let $\mathcal{E}$ be the exponential martingale of the martingale $M_t=\int_0^tb_s\nabla W_s$, and define
$$\gamma_t=\exp(\int_0^ta_s\nabla s), ~\Gamma_t=\gamma_t\mathcal{E}_t$$
Suppose that
\begin{itemize}
\item $\mathcal{E}_t$ is a positive uniformly integrable martingale;
\item $E[(sup_t\gamma_t)\mathcal{E}_T^2]<\infty$.
\end{itemize}
If the linear backward equation
\begin{equation}\nonumber
-\nabla Y_t=(a_tY_{t-}+b_tZ_t+c_t)\nabla t-Z_t\nabla W_t-\nabla N_t,~Y_T=\xi
\end{equation}
has a solution $(Y,Z,N)$, then $Y$ is given by
\begin{equation}\nonumber
Y_t=E[\xi\frac{\Gamma_T}{\Gamma_t}+\int_t^T\frac{\Gamma_s}{\Gamma_t}c_s\nabla s|\mathcal{F}_t].
\end{equation}
\end{lemma}
\begin{proof}
By Girsanov's theorem \cite{girsanov},
$$\tilde{W}_t=W_t-\int_0^tb_s\nabla s$$
is a Q-martingale, the probability measure Q which has density $\mathcal{E}_T$ with respect to $\mathbf{P}$ on $\mathcal{F}_T$.
Under the measure Q, the equation becomes
$$-\nabla Y_t=(a_tY_{t-}+c_t)\nabla t-Z_t\nabla \tilde{W}_t-\nabla N_t,~Y_T=\xi.$$
Applying It\^{o} formula to the process $\gamma Y$, we obtain
\begin{equation}\nonumber
\gamma_tY_t=\gamma_0Y_0+\int_0^t\gamma_sZ_s\nabla \tilde{W}_s+\int_0^t\gamma_s\nabla N_s-\int_0^tc_s\gamma_s\nabla s+ S
\end{equation}
where $S=\sum_{s\in\left(0,t\right]}(\gamma_{s}Y_{s}-\gamma_{s}Y_{s-}-(\gamma_{s}(Y_{s}-Y_{s-})))=0$.
Since $\int_0^t\gamma_sZ_s\nabla \tilde{W}_s$ and $\int_0^t\gamma_s\nabla N_s$ are martingales,
\begin{equation}\nonumber
\gamma_tY_t=E^Q[\xi\gamma_T+\int_t^Tc_s\gamma_s\nabla s|\mathcal{F}_t],
\end{equation}
and consequently
\begin{equation}\nonumber
\Gamma_tY_t=E[\xi\Gamma_T+\int_t^Tc_s\Gamma_s\nabla s|\mathcal{F}_t].
\end{equation}
This concludes the proof.
\end{proof}

We have the one-dimension comparison result
\begin{theorem}
Let $(Y^1,Z^1,N^1)$, $(Y^2,Z^2,N^2)$ be the corresponding solutions to $(g^1(t,Y_{t-}^1,Z_t^1),\xi^1)$, $(g^2(t,Y_{t-}^2,Z_t^2),\xi^2)$ and suppose $g^1,g^2$ satisfy conditions (H2)(H3).
Assume that $\xi_1\geq \xi_2$ and satisfy (H4), $g^1(t,Y_{t-}^2,Z_{t}^2)\geq g^2(t,Y_{t-}^2,Z_{t}^2)~a.s.-\mathbf{P}$.
Then, for any $t$, we have $Y_t^1\geq Y_t^2~a.s.-\mathbf{P}$.
\end{theorem}
\begin{proof}
We define $\delta Y_t=Y_t^1-Y_t^2$, $\delta Z_t=Z_t^1-Z_t^2$, $\delta N_t=N_t^1-N_t^2$.
Then $\delta Y$ solves the following linear BSDE on time scales
\begin{equation}\nonumber
-\nabla \delta Y_t=(a_t\delta Y_{t-}+b_t\delta Z_t+c_t)\nabla t-\delta Z_t\nabla W_t-\nabla \delta N_t,~\delta Y_T=\xi^1-\xi^2,
\end{equation}
where
$$a_t=\frac{g^1(t,Y_t^1,Z_t^1)-g^1(t,Y_t^2,Z_t^1)}{\delta Y_t}1_{\delta Y_t\neq 0},$$
$$b_t=\frac{g^1(t,Y_t^2,Z_t^1)-g^1(t,Y_t^2,Z_t^2)}{\delta Y_t}1_{\delta Z_t\neq 0},$$
$$c_t=g^1(t,Y^2_{t-},Z_t^2)-g^2(t,Y_{t-}^2,Z_{t}^2).$$
According to the linear BSDE on time scales, if the coefficients $a_t,b_t,c_t$ verify the suitable conditions, we have
\begin{equation}\nonumber
\Gamma_t\delta Y_t=E[\Gamma_T(\xi^1-\xi^2)+\int_t^T\delta c_s\Gamma_s\nabla s|\mathcal{F}_t].
\end{equation}
Assume now that $\xi^1\geq \xi^2$, and for any $t$, $c_t\geq 0~a.s.-\mathbf{P}$.
Then, for any $t$, $Y_t^1\geq Y_t^2~a.s.-\mathbf{P}$.
\end{proof}

\end{document}